\documentclass[reqno]{amsart}
\usepackage{mathrsfs}
\usepackage{color}
\usepackage{amsmath}
\usepackage{amsfonts}
\usepackage{amssymb}
\usepackage{graphicx}%
\usepackage{hyperref}

 \newtheorem{Theorem}{Theorem}[section]
 
 \newtheorem{Lemma}{Lemma}[section]
 \newtheorem{Proposition}{Proposition}[section]

 \newtheorem{Definition}{Definition}[section]
\newtheorem{Question}{Question}[section]

 \newtheorem{Remark}{Remark}[section]

 \numberwithin{equation}{section}

\begin{document}

\title[Decreasing equisingular approximations with analytic singularities]
 {Decreasing equisingular approximations with analytic singularities}

\author{Qi'an Guan}
\address{Qi'an Guan: School of Mathematical Sciences,
Peking University, Beijing, 100871, China.}
\email{guanqian@math.pku.edu.cn}

\thanks{}

\subjclass[2010]{32D15, 32E10, 32L10, 32U05}

\keywords{multiplier ideal sheaf, plurisubharmonic function, equisingular approximation, analytic singularity}

\date{\today}

\dedicatory{}

\commby{}

%%% ----------------------------------------------------------------------

\begin{abstract}
In this note, for the multiplier ideal sheaves with weights $\log\sum_{i}|z_{i}|^{a_{i}}$,
we present the sufficient and necessary condition of the existence of
decreasing equisingular approximations with analytic singularities.
\end{abstract}

%%% ----------------------------------------------------------------------
\maketitle
%%% ----------------------------------------------------------------------

\section{Introduction}

The multiplier ideal sheaf related to a plurisubharmonic function plays an important role in complex geometry and algebraic geometry,
which was widely discussed
(see e.g. \cite{tian87,Nadel90,siu96,DEL00,D-K01,demailly-note2000,D-P03,Laz04,siu05,siu09,demailly2010}).
We recall the definition as follows.

\emph{Let $\varphi$ be a plurisubharmonic function (see \cite{demailly-book,Si,siu74}) on a complex manifold.
It is known that the multiplier ideal sheaf $\mathcal{I}(\varphi)$ was defined as the sheaf of germs of holomorphic functions $f$ such that
$|f|^{2}e^{-2\varphi}$ is locally integrable (see \cite{demailly2010}).}

Recall that a weight was called with analytic singularity near $o$ if
$$\varphi=c\log\sum_{j=1}^{N}|f_{j}|^{2}+O(1)$$
near $o$, where $\{f_{j}\}_{j=1}^{N}$ are holomorphic functions near $o$ and $c\in\mathbb{R}^{+}$.

Demailly established the decreasing equisingular approximations $\varphi_{m}$ of weight $\varphi$, which are smooth outside analytic subvarieties (see \cite{demailly2010}),
where "equisingular" means $\mathcal{I}(\varphi)=\mathcal{I}(\varphi_m)$ holds for any $m$.
Then it was asked: can one choose the decreasing equisingular approximations with analytic singularities (see \cite{demailly2010})?

In \cite{guan-approx}, by constructing an example and using
the sharp lower bound of log canonical threshold \cite{D-P14} to check the example,
a negative answer to the above question was presented.
Furthermore, it is natural to ask

\begin{Question}
\label{Q:approx}
For a (given) class of weights,
can one give a characterization of the weights which have
decreasing equisingular approximations with analytic singularities?
\end{Question}

Let $(z_{1},\cdots,z_{n})$ be the coordinates on $\mathbb{C}^{n}$.
In this article,
we answer Question \ref{Q:approx} for the class of weights $\{\log\sum_{i=1}^{m}|z_{i}|^{a_{i}}:$
$m\leq n$ and $a_{i}>0$ for any $1\leq i\leq m\}$.

\begin{Theorem}
\label{thm:approx}
The weight $\varphi=\log\sum_{i=1}^{m}|z_{i}|^{a_{i}}$ has
decreasing equisingular approximations with analytic singularities near $o$ if and only if one of the following statements holds

$(1)$ $\varphi$ has analytic singularity near $o$, i.e.
there exists $c\in\mathbb{R}^{+}$ such that $\frac{a_{i}}{c}\in \mathbb{Q}^{+}$ for any $i\in\{1,\cdots,m\}$;

$(2)$ the equation $\sum_{i=1}^{m}\frac{x_{i}}{a_{i}}=1$ has no positive integer solutions.
\end{Theorem}

Considering the complex lines through $o$, one can obtain the equivalence between "$\varphi$ has analytic singularity near $o$" and "there exists $c\in\mathbb{R}^{+}$ such that $\frac{a_{i}}{c}\in \mathbb{Q}^{+}$ for any $i\in\{1,\cdots,m\}$" in Theorem \ref{thm:approx}.

\begin{Remark}
Note that if $c\not\in\mathbb{Q}^{+}$, then statement $(2)$ in Theorem \ref{thm:approx} holds.
Then statement $(1)$ in Theorem \ref{thm:approx} can be replaced by

"$(1')$ $a_{i}\in \mathbb{Q}^{+}$ for any $i\in\{1,\cdots,m\}$;".
\end{Remark}

\section{Preparations}

In the present section, we recall and present some results which will be used to prove Theorem \ref{thm:approx}.

\subsection{The multiplier ideal with weight $\log\max_{1\leq i\leq m}|z_{i}|^{a_{i}}$}

\

Let $\varphi=\log\max_{1\leq i\leq m}|z_{i}|^{a_{i}}$, where $m\leq n$ and $a_{i}>0$ for any $1\leq i\leq m$.
Let $V_{r_{0}}=\{\max\{|z_{m+1}|,\cdots,|z_{n}|\}<\frac{1}{2}\log r_{0}\}$.
\begin{Lemma}
\label{lem:torus}
Let $f=\sum_{\alpha}b_{\alpha}z^{\alpha}$ (Taylor expansion) be a holomorphic function on a neighborhood $U\ni o$, such that $|f|^{2}e^{-2\varphi}$ is integrable on $U$.
Then for any $\alpha$,
$$\int_{\{\varphi<\frac{1}{2}\log r\}\cap V_{r_{0}}}|f|^{2}e^{-2\varphi}
\geq\int_{\{\varphi<\frac{1}{2}\log r\}\cap V_{r_{0}}}|b_{\alpha}z^{\alpha}|^{2}e^{-2\varphi}$$
holds for any $r>0$ and $r_{0}>0$ small enough such that $\{\varphi<\frac{1}{2}\log r\}\cap V_{r_{0}}\subset U$,
which implies that
$|a_{\alpha}z^{\alpha}|^{2}e^{-2\varphi}$ is integrable near $o$ for any $\alpha$.
\end{Lemma}

It follows from Lemma \ref{lem:torus} that the following lemma holds.

\begin{Lemma}
\label{lem:torus_compara}
The following two statements are equivalent

(1) $\mathcal{I}(t_{1}\varphi)_{o}=\mathcal{I}(t_{2}\varphi)_{o}$;

(2) $\mathcal{I}(t_{1}\varphi)_{o}\cap\{(z^{\alpha},o)\}_{\alpha\in\mathbb{N}^{n}}=\mathcal{I}(t_{2}\varphi)_{o}\cap\{(z^{\alpha},o)\}_{\alpha\in\mathbb{N}^{n}}$.
\end{Lemma}

There is a standard calculation as follows
\begin{equation}
\label{equ:standard_int}
\begin{split}
&\int_{\{\varphi<\frac{1}{2}\log r\}\cap V_{r_{0}}}|z_{1}^{\alpha_{1}}\cdots z_{n}^{\alpha_{n}}|^{2}
\\=&
\int_{\{|z_{1}|<r^{\frac{1}{2a_{1}}}\}\cap\cdots\cap\{|z_{m}|<r^{\frac{1}{2a_{m}}}\}\cap V_{r_{0}}}|z_{1}^{\alpha_{1}}\cdots z_{n}^{\alpha_{n}}|^{2}
\\=&
\int_{\{|z_{1}|<r^{\frac{1}{2a_{1}}}\}\cap\cdots\cap\{|z_{m}|<r^{\frac{1}{2a_{m}}}\}}|z_{1}^{\alpha_{1}}\cdots z_{m}^{\alpha_{m}}|^{2}
\\&\times\int_{V_{r_{0}}}|z_{m+1}^{\alpha_{m+1}}\cdots z_{n}^{\alpha_{n}}|^{2}
\\=&\pi^{m}\frac{r^{\sum_{i=1}^{m}\frac{\alpha_{i}+1}{a_{i}}}}{\prod_{i=1}^{m}(\alpha_{i}+1)}\int_{V_{r_{0}}}|z_{m+1}^{\alpha_{m+1}}\cdots z_{n}^{\alpha_{n}}|^{2},
\end{split}
\end{equation}
which implies the following lemma.
\begin{Lemma}
\label{lem:ideal}
$(z_{1}^{\alpha_{1}}\cdots z_{n}^{\alpha_{n}},o)\in \mathcal{I}(\varphi)_{o}$
if and only if
$\sum_{i=1}^{m}\frac{\alpha_{i}+1}{a_{i}}>1$.
\end{Lemma}

\begin{Remark}
\label{rem:ideal}
Assume that the equation $\sum_{i=1}^{m}\frac{x_{i}}{a_{i}}=1$ has no integer solutions.
Then there exists $\varepsilon_{0}>0$, such that for any $\varepsilon\in(0,\varepsilon_{0}]$
\begin{equation}
\label{equ:acc}
\mathcal{I}((1-\varepsilon)\varphi)_{o}=\mathcal{I}(\varphi)_{o}
\end{equation}
holds.
\end{Remark}

\begin{proof}
Note that there for any $a_{i}>0$ for any $1\leq i\leq m$,
there exist finite $\alpha\in\mathbb{N}^{m}$,
such that $\sum_{i=1}^{m}\frac{\alpha_{i}+1}{a_{i}}\leq 1$.
As the equation $\sum_{i=1}^{m}\frac{x_{i}}{a_{i}}=1$ has no integer solutions,
then there exists $\varepsilon_{0}>0$,
such that for any $\varepsilon\in(0,\varepsilon_{0}]$, the equality
\begin{equation}
\{(x_{i})_{1\leq i\leq m}:\sum_{i=1}^{m}\frac{x_{i}}{a_{i}}>1-\varepsilon\,\&\,x_{i}\in\mathbb{N}^{+}\}=
\{(x_{i})_{1\leq i\leq m}:\sum_{i=1}^{m}\frac{x_{i}}{a_{i}}>1\,\&\,x_{i}\in\mathbb{N}^{+}\}
\end{equation}
holds.
Combining with \ref{lem:ideal},
one can obtain that
$$\mathcal{I}((1-\varepsilon)\varphi)_{o}\cap\{(z^{\alpha},o)\}_{\alpha\in\mathbb{N}^{n}}
=\mathcal{I}(\varphi)_{o}\cap\{(z^{\alpha},o)\}_{\alpha\in\mathbb{N}^{n}}$$
Then it follows from
Lemma \ref{lem:torus_compara} that the following remark holds.
\end{proof}

\subsection{Maximal equisingular weights and minimal integrations}

One can define the maximal equisingular weight as follows

\begin{Definition}
\label{def:maximal_germ}
A weight $\varphi_{\max}$ was called maximal equisingular near $z_{0}$
if for any plurisubharmonic function $\varphi\geq \varphi_{\max}+O(1)$ near $z_{0}$
satisfying $\mathcal{I}(\varphi)_{z_{0}}=\mathcal{I}(\varphi_{\max})_{z_{0}}$,
the inequality $\varphi=\varphi_{\max}+O(1)$ holds near $z_{0}$.
\end{Definition}

It follows from Definition \ref{def:maximal_germ} that for any plurisubharmonic function $\varphi$ near $z_{0}$,
\begin{equation}
\label{equ:maximal}
\mathcal{I}(\varphi_{\max})_{z_{0}}=\mathcal{I}(\max\{\varphi,\varphi_{\max}+C\})_{z_{0}}
\end{equation}
holds for any $C>0$.

\begin{Remark}
\label{rem:maximal_analytic}
Let $\varphi$ be a maximal equisingular weight near $o$,
Then the following two statements are equivalent

(1) $\varphi$ has decreasing equisingular approximations with analytic singularities near $o$;

(2) $\varphi$ has analytic singularities near $o$.
\end{Remark}

Let $F$ be a holomorphic function on pseudoconvex domain $D\subset\mathbb{C}^{n}$ (see \cite{demailly-book}) containing the origin $o\in\mathbb{C}^{n}$.

Recall that the minimal integration related to ideal $I\subset\mathcal{O}_{o}$ was defined in \cite{guan-effect} as follows $$C_{F,I}(D):=\inf\{\int_{D}|\tilde{F}|^{2}|(\tilde{F}-F,o)\in I\&\tilde{F}\in\mathcal{O}(D)\}.$$

Let $\varphi$ be a plurisubharmonic function on $D$.
In \cite{guan-effect}, the following concavity of $G(-\log r):=C_{F,\mathcal{I}(\varphi)_{o}}(\{\varphi<\frac{1}{2}\log r\})$ was presented.

\begin{Proposition}
\label{prop:logconcave}\cite{guan-effect}
Assume that $G(0)<+\infty$.
Then
$G(-\log r)$ is concave with respect to $r\in(0,1]$,
which implies that

(1) the inequality $G(-\log r)\geq rG(0)$ holds for any $r\in(0,1]$;

(2) the equality $G(-\log r)=rG(0)$
holds for some $r\in(0,1)$ if and only if the equality holds for any $r\in(0,1]$.
\end{Proposition}

It follows from the dominated convergence theorem that $G(0)<+\infty$ implies $\lim_{r\to0+0}G(-\log r)=0$.

\begin{Lemma}
\label{lem:compara}
Let $\varphi_{1}\leq\varphi_{2}< 0$ be plurisubharmonic functions on $D$.
Assume that $C_{F,\mathcal{I}(\varphi_{1})_{o}}(D)<+\infty$, and $(F,o)\not\in \mathcal{I}(\varphi_{1})_{o}$.
Then the following three statements are equivalent

(1) $\varphi_{1}=\varphi_{2}$ on $D$;

(2) the quality $\{\varphi_{1}<\log r\}=\{\varphi_{2}<\log r\}$ holds for any $r\in(0,1]$;

(3) the equality $C_{F,\mathcal{I}(\varphi_{1})_{o}}(\{\varphi_{1}<\log r\})=C_{F,\mathcal{I}(\varphi_{1})_{o}}(\{\varphi_{2}<\log r\})$ holds for any $r\in(0,1]$.
\end{Lemma}

\begin{proof}
It suffices to prove that $(3)\Rightarrow (2)$. We prove it by contradiction: if not,
then $\{\varphi_{1}<\log r\}\subsetneq\{\varphi_{2}<\log r\}$ holds for some $r$.
It follows from $C_{F,\mathcal{I}(\varphi_{1})_{o}}(D)<+\infty$, $(F,o)\not\in \mathcal{I}(\varphi_{1})_{o}$ and Lemma 2.2 in \cite{guan-effect} there exists (unique) $\tilde{F}\in\mathcal{O}(\{\varphi_{1}<\log r\})$ such that $\int_{\{\varphi_{1}<\log r\}}|\tilde{F}|^{2}=C_{F,\mathcal{I}(\varphi_{1})_{o}}(\{\varphi_{1}<\log r\})$ and $(\tilde{F}-F,o)\in\mathcal{I}(\varphi_{1})_{o}$.
Then we have
\begin{equation}
\label{equ:compara}
\begin{split}
C_{F,\mathcal{I}(\varphi_{1})_{o}}(\{\varphi_{1}<\log r\})
&=\int_{\{\varphi_{1}<\log r\}}|\tilde{F}|^{2}
\\&>\int_{\{\varphi_{2}<\log r\}}|\tilde{F}|^{2}\geq
C_{F,\mathcal{I}(\varphi_{1})_{o}}(\{\varphi_{2}<\log r\}),
\end{split}
\end{equation}
which contradicts statement $(3)$ in the present Lemma.
\end{proof}

By Proposition \ref{prop:logconcave} and Lemma \ref{lem:compara}, one can obtain the following remark.

\begin{Remark}
\label{rem:equality_lower_optimal}
If equality $G(-\log r)=rG(0)$ holds for some $r\in(0,1)$,
then for any plurisubharmonic function $\varphi_{2}\geq\varphi$ on $D$ satisfying
$\mathcal{I}(\varphi_{2})_{z_{0}}=\mathcal{I}(\varphi)_{z_{0}}$,
the equality $\varphi_{2}=\varphi$ holds on $D$.
\end{Remark}

\begin{proof}
Proposition \ref{prop:logconcave} shows that
if equality $G(-\log r)=rG(0)$ holds for some $r\in(0,1)$,
then  $G(-\log r)=rG(0)$ holds for any $r\in(0,1)$.
By Lemma \ref{lem:compara} $(\varphi\sim \varphi_{1})$ and $\varphi_{2}\geq\varphi$,
it follows that
$\varphi_{2}=\varphi$ holds on $D$.
\end{proof}

\begin{Lemma}
\label{lem:max_equisingular}
Let $\varphi_{1}$ and $\varphi_{2}$ be two plurisubharmonic functions near $z_{0}$,
such that $\varphi_{1}\leq\varphi_{2}+O(1)$.
If $\mathcal{I}(\varphi_{1})_{z_{0}}=\mathcal{I}(\varphi_{2})_{z_{0}}$,
then $$\mathcal{I}(\varphi_{1})_{z_{0}}=\mathcal{I}(\varphi_{2})_{z_{0}}=\mathcal{I}(\max\{\varphi_{1},\varphi_{2}+C\})_{z_{0}}$$
holds for any $C\in\mathbb{R}$.
\end{Lemma}

\begin{proof}
Note that $\varphi_{1}\leq\max\{\varphi_{1},\varphi_{2}+C\}\leq \varphi_{2}+O(1)$
holds near $z_{0}$.
\end{proof}

Remark \ref{rem:equality_lower_optimal} implies the following sufficient condition of maximal equisingular weight.

\begin{Proposition}
\label{pro:maximal_sufficient}
Let $\{D_{j}\}_{j=1,2,\cdots}$ be a sequence of pseudoconvex subdomains on $D$
satisfying $\cap_{j}D_{j}=\{o\}$, $\varphi|_{D_{j}}<\frac{1}{2}\log r_{j}$.
Let $F_{j}$ be holomorphic functions on $D_{j}$.
If
\begin{equation}
\label{equ:equiv}
\frac{1}{r'_{j}}\int_{\{\varphi<\frac{1}{2}\log r'_{j}\}}|F_{j}|^{2}=\frac{C_{F,\mathcal{I}(\varphi)_{o}}(D_{j})}{r_{j}}
\end{equation}
holds for any $j\in\{1,2,\cdots\}$ and some $r'_{j}\in(0,r_{j})$,
then $\varphi$ is a maximal equisingular weight near $o$.
\end{Proposition}

\begin{proof}
We prove Proposition \ref{pro:maximal_sufficient} by contradiction:
if not, there exists a weight $\varphi'$ on a neighborhood $U\ni o$ such that

(1) $\varphi'\geq \varphi$ and $\varlimsup_{z\to o}(\varphi'(z)-\varphi(z))=+\infty$;

(2) $\mathcal{I}(\varphi')_{o}=\mathcal{I}(\varphi)_{o}$.

One can find some $j$ such that $D_{j}\subset\subset U$.
Then there exists $N_{j}>>0$ such that $\varphi'-N_{j}<\frac{1}{2}\log r_{j}$ on $D_{j}$.

Consider $\max\{\varphi,\varphi'-N_{j}\}<\frac{1}{2}\log r_{j}$.
It follows from Lemma \ref{lem:max_equisingular} that
$\mathcal{I}(\max\{\varphi,\varphi'-N_{j}\})_{o}=\mathcal{I}(\varphi)_{o}$.
Note that Remark \ref{rem:equality_lower_optimal} $(\max\{\varphi,\varphi'-N_{j}\}-\frac{1}{2}\log r_{j}\sim \varphi_{2},$
$\varphi-\frac{1}{2}\log r_{j}\sim \varphi$, $\frac{r'_{j}}{r_{j}}\sim r)$ implies that
$$\max\{\varphi,\varphi'-N_{j}\}=\varphi,$$
which contradicts (1) $\varlimsup_{z\to o}(\varphi'(z)-\varphi(z))=+\infty$.
Then Proposition \ref{pro:maximal_sufficient} has been proved.
\end{proof}

Equality \ref{equ:standard_int} implies the following lemma.

\begin{Lemma}
\label{lem:minimal_int}
Let $\varphi=\log\max_{1\leq i\leq m}|z_{i}|^{a_{i}}$ $(m\leq n)$, where $a_{i}>0$ for any $1\leq i\leq m$.
Let $\alpha=(\alpha_{1},\cdots,\alpha_{n})$.
Assume that $\sum_{i=1}^{m}\frac{\alpha_{i}+1}{a_{i}}\leq 1$.
Then
\begin{equation}
\label{equ:minimal_int}
\begin{split}
C_{z^{\alpha},\mathcal{I}(\varphi)_{o}}(\{\varphi<\frac{1}{2}\log r\}\cap V_{r_{0}})
&=\int_{\{\varphi<\frac{1}{2}\log r\}\cap V_{r_{0}}}|z^{\alpha}|^{2}
\\&=
\pi^{m}\frac{r^{\sum_{i=1}^{m}\frac{\alpha_{i}+1}{a_{i}}}}{(\alpha_{1}+1)\cdots (\alpha_{m}+1)}\int_{V_{r_{0}}}|z_{m+1}^{\alpha_{m+1}}\cdots z_{n}^{\alpha_{n}}|^{2}.
\end{split}
\end{equation}
\end{Lemma}

\begin{proof}
Note that for any $L^{2}$ integrable holomorphic function $z^{\alpha}+\sum_{\alpha'\neq\alpha}b_{\alpha'}z^{\alpha'}$ on $\{\varphi<\frac{1}{2}\log r\}\cap V_{r_{0}}$,
inequality
\begin{equation}
\label{equ:minimal_ideal}
\int_{\{\varphi<\frac{1}{2}\log r\}\cap V_{r_{0}}}|z^{\alpha}+\sum_{\alpha'\neq\alpha}b_{\alpha'}z^{\alpha'}|^{2}\geq\int_{\{\varphi<\frac{1}{2}\log r\}\cap V_{r_{0}}}|z^{\alpha}|^{2}
\end{equation}
holds,
then one can obtain the present lemma by equality \ref{equ:standard_int}.
\end{proof}

If $\sum_{i=1}^{m}\frac{\alpha_{i}+1}{a_{i}}>1$,
then it follows from Lemma \ref{lem:ideal} that $$(z^{\alpha},o)\in\mathcal{I}(\varphi)_{o},$$
which implies that
$$C_{z^{\alpha},\mathcal{I}(\varphi)_{o}}(\{\varphi<\frac{1}{2}\log r\}\cap V_{r_{0}})=0$$
holds for any $r>0$ and $r_{0}>0$.

It follows from Proposition \ref{pro:maximal_sufficient} and Lemma \ref{lem:minimal_int} that

\begin{Remark}
\label{rem:maximal}
Assume that
$\sum_{i=1}^{m}\frac{\alpha_{i}+1}{a_{i}}=1$ for some $\alpha\in\mathbb{N}^{m}$.
Then $\varphi=\log\max_{1\leq i\leq m}|z_{i}|^{a_{i}}$ is a maximal weight,
where $a_{i}>0$ for any $1\leq i\leq m$.
\end{Remark}

\begin{Proposition}
\label{pro:equiv_maximal}
Let $\varphi=\log\max_{1\leq i\leq m}|z_{i}|^{a_{i}}$, where $a_{i}>0$ $(1\leq i\leq m)$, and $m\leq n$.
Then the following three statements are equivalent

(1) $\varphi$ is not a maximal equisingular weight near $o$;

(2) there exists $\varepsilon\in(0,1)$ such that $\mathcal{I}(\varphi)_{o}=\mathcal{I}((1-\varepsilon)\varphi)_{o}$;

(3) the equation $\sum_{i=1}^{m}\frac{x_{i}}{a_{i}}=1$ has no positive integer solutions.
\end{Proposition}

\begin{proof}
It follows from Definition \ref{def:maximal_germ} and Remark \ref{rem:ideal} (respectively) that $"(2)\Rightarrow (1)"$ and $"(3)\Rightarrow (2)"$ hold (respectively).
Then it suffices to consider $"(1)\Rightarrow (3)"$.

We prove $"(1)\Rightarrow (3)"$ by contradiction:
if not, then
the equation $\sum_{i=1}^{m}\frac{x_{i}}{a_{i}}=1$ has a positive integer solution denoted by $(\alpha_{1}+1,\cdots,\alpha_{m}+1)$.
It follows from Remark \ref{rem:maximal} that $\varphi$ is a maximal equisingular weight near $o$, which contradicts statement $(1)$.
Then we obtain that $(1)\Rightarrow (3)$ holds.
\end{proof}

\subsection{Decreasing equisingular approximations with analytic singularities}
\

\begin{Lemma}
\label{lem:approxi}
Let $\varphi=\log\max_{1\leq i\leq m}|z_{i}|^{a_{i}}$, where $a_{i}>0$ $(1\leq i\leq m)$, and $m\leq n$.
If there exists $\varepsilon\in(0,1)$ such that $\mathcal{I}(\varphi)_{o}=\mathcal{I}((1-\varepsilon)\varphi)_{o}$,
then $\mathcal{I}(\varphi)_{o}$ has decreasing equisingular approximations with analytic singularities near $o$.
\end{Lemma}

\begin{proof}
Note that for any $a_{i}$ there exists rational numbers $a_{i,k}\in((1-\varepsilon)a_{i},a_{i})$ such that $\lim_{k\to+\infty}a_{i,k}=a_{i}$
holds for any $i\in\{1,\cdots,m\}$,
and increasing with respect to $k$.
Let $\varphi_{k}=\log\max_{1\leq i\leq m}|z_{i}|^{a_{i,k}}$.
Then the sequence $\{\varphi_{k}\}_{k}$ with analytic singularities near $o$ are decreasing convergent to $\varphi$ with respect to $k$ on $\{z:\max_{1\leq i\leq m}|z_{i}|<1\}$.

Note that $\varphi\leq \varphi_{k}\leq (1-\varepsilon)\varphi$ near $o$ for any $k$,
which implies
$$\mathcal{I}(\varphi)_{o}\subseteq\mathcal{I}(\varphi_{k})_{o}\subseteq\mathcal{I}((1-\varepsilon)\varphi)_{o}$$
holds for any $k$.
Then it follows from $\mathcal{I}(\varphi)_{o}=\mathcal{I}((1-\varepsilon)\varphi)_{o}$
that
$$\mathcal{I}(\varphi)_{o}=\mathcal{I}(\varphi_{k})_{o}$$
holds for any $k$.
Combining with the decreasing property of $\varphi_{k}$,
we obtain the present lemma.
\end{proof}

\section{Proof of Theorem \ref{thm:approx}}

We prove Theorem \ref{thm:approx} by two steps: sufficiency and necessity respectively.

\

Step 1. (Sufficiency).
It suffices to prove that $(2)$ implies that $\varphi$ has decreasing equisingular approximations with analytic singularities.
By Proposition \ref{pro:equiv_maximal},
it follows that (2) implies that there exists $\varepsilon\in(0,1)$ such that $\mathcal{I}(\varphi)_{o}=\mathcal{I}((1-\varepsilon)\varphi)_{o}$.
By Lemma \ref{lem:approxi}, we obtain that $(2)$ implies that $\varphi$ has decreasing equisingular approximations with analytic singularities.
Then the sufficiency has been proved.

\

Step 2. (Necessity). We prove by contradiction: if not, i.e. neither statement $(1)$ nor statement $(2)$ holds, and $\varphi$ has decreasing equisingular approximations with analytic singularities,
then it follows from Proposition \ref{pro:equiv_maximal} that $\varphi$ is a maximal weight near $o$.
By Remark \ref{rem:maximal_analytic}, it follows from $\varphi$ has decreasing equisingular approximations with analytic singularities that
$\varphi$ has analytic singularity near $o$, which contradicts the assumption that statement $(1)$ does not hold. Then the necessity has been proved.

\vspace{.1in} {\em Acknowledgements}. The author would like to
thank Professor Jean-Pierre Demailly, Professor Takeo Ohsawa and Professor Xiangyu Zhou for helpful discussions and encouragements.
The author would also like to
thank the hospitality of Beijing International Center for Mathematical Research.
The author was partially supported by by NSFC-11522101 and NSFC-11431013.

\bibliographystyle{references}
\bibliography{xbib}

\begin{thebibliography}{100}
\bibitem{cao2014}J.Y. Cao, Numerical dimension and a Kawamata-Viehweg-Nadel-type vanishing theorem on compact K\"{a}hler manifolds. Compos. Math. 150 (2014), no. 11, 1869--1902.
\bibitem{demailly-note2000}J-P. Demailly, Multiplier ideal sheaves and analytic methods in algebraic geometry. School on Vanishing Theorems and Effective Results in Algebraic Geometry (Trieste, 2000), 1--148, ICTP Lect. Notes, 6, Abdus Salam Int. Cent. Theoret. Phys., Trieste, 2001.
\bibitem{demailly2010}J.-P. Demailly, Analytic Methods in Algebraic Geometry, Higher Education Press, Beijing, 2010.
\bibitem{demaillyAbel}J.-P. Demailly, On the cohomology of pseudoeffective line bundles. Complex geometry and dynamics, 51--99, Abel Symp., 10, Springer, Cham, 2015.
\bibitem{demailly-book}J.-P. Demailly, Complex analytic and differential geometry, electronically accessible
at http://www-fourier.ujf-grenoble.fr/~demailly/books.html.
\bibitem{DEL00}J-P. Demailly, L. Ein, and R. Lazarsfeld, A subadditivity property of multiplier
ideals, Michigan Math. J. 48 (2000), 137--156.
\bibitem{D-K01}J-P. Demailly, J. Koll\'{a}r,
Semi-continuity of complex singularity exponents and K\"{a}hler-Einstein metrics on Fano orbifolds.
Ann. Sci. \'{E}cole Norm. Sup. (4) 34 (2001), no. 4, 525--556.
\bibitem{D-P03}J-P. Demailly, T. Peternell, A Kawamata-Viehweg vanishing theorem on compact K\"{a}hler manifolds. J. Differential Geom.  63  (2003),  no. 2, 231--277.
\bibitem{D-P14}J.-P. Demailly, H. Pham, A sharp lower bound for the log canonical threshold. Acta Math. 212 (1), 1¨C9 (2014).
\bibitem{guan-approx}Q.A. Guan, Nonexistence of decreasing equisingular approximations with logarithmic poles. J. Geom. Anal. 27 (2017), no. 1, 886--892.
\bibitem{guan-effect}Q.A. Guan, A sharp effectiveness result of Demailly's strong openness conjecture,  arXiv:1709.05880v4 [math.CV].
\bibitem{GZopen-c}Q.A. Guan and X.Y. Zhou, A proof of Demailly's strong openness conjecture,
Ann. of Math. (2) 182 (2015), no. 2, 605--616. See also arXiv:1311.3781.
\bibitem{Laz04}Lazarsfeld R. Positivity in algebraic geometry. I. Classical setting: line bundles and linear series;
II. Positivity for vector bundles, and multiplier ideals. Ergebnisse der Mathematik und ihrer Grenzgebiete. 3. Folge.
 A Series of Modern Surveys in Mathematics, 48, 49. Springer-Verlag, Berlin, 2004.
\bibitem{Lempert14}L. Lempert, Modules of square integrable holomorphic germs. Analysis meets geometry, 311--333, Trends Math., Birkh\"{a}user/Springer, Cham, 2017.
\bibitem{Nadel90}A. Nadel, Multiplier ideal sheaves and K\"{a}hler-Einstein metrics of positive scalar curvature.
Ann. of Math. (2) 132 (1990), no. 3, 549--596.
\bibitem{Si}Nessim Sibony: Quelques probl\`emes de prolongement de courants en
analyse complexe. (French) [Some extension problems for currents in
complex analysis] Duke Math. J. 52 (1985), no. 1, 157--197.
\bibitem{siu74}Y.T. Siu, Analyticity of sets associated to Lelong numbers and the extension of closed positive currents. Invent. Math. 27 (1974), 53--156.
\bibitem{siu96}Y.T. Siu, The Fujita conjecture and the extension theorem of Ohsawa-Takegoshi,
Geometric Complex Analysis, Hayama. World Scientific (1996),
577--592.
\bibitem{siu05}Y.T. Siu, Multiplier ideal sheaves in complex and algebraic geometry. Sci. China Ser. A 48
(2005), suppl., 1--31.
\bibitem{siu09} Y.T. Siu, Dynamic multiplier ideal sheaves and the construction of rational curves in Fano
manifolds. Complex analysis and digital geometry, 323--360, Acta Univ. Upsaliensis Skr.
Uppsala Univ. C Organ. Hist., 86, Uppsala Universitet, Uppsala, 2009.
\bibitem{tian87}G. Tian, On K\"{a}hler-Einstein metrics on certain K\"{a}hler manifolds with $C_{1}(M)>0$, Invent. Math.  89  (1987),  no. 2, 225--246.
\end{thebibliography}

\end{document}